\documentclass{amsart}

\usepackage{amssymb,amsmath,latexsym,amsthm}
\usepackage{MnSymbol}
\usepackage{graphicx}
\usepackage{fontenc}%
\usepackage{mathrsfs}
\usepackage{color}

\usepackage[verbose]{wrapfig}
\usepackage[english]{babel} 
\usepackage{pstricks-add} 

\newtheorem{theorem}{Theorem}[section]
\newtheorem{definition}[theorem]{Definition}
\newtheorem{lemma}[theorem]{Lemma}

\newtheorem{example}[theorem]{Example}

\newtheorem{remark}[theorem]{Remark}

\usepackage{pstricks,pst-text,pst-grad,pst-node,pst-3dplot,pstricks-add,pst-poly}
\usepackage{pst-solides3d}
\definecolor{pink}{rgb}{1, .75, .8}
\psset{arrows=->, labelsep=3pt, mnode=circle}

\definecolor{lgrey}{gray}{.85}% \colorbox{lgrey}{}

\def\defineTColor#1#2{%
 \newpsstyle{#1}{%
  fillstyle=vlines,hatchcolor=#2,
  hatchwidth=0.1\pslinewidth,
  hatchsep=1\pslinewidth}%
  }
\defineTColor{Tgray}{gray}
\defineTColor{Tgreen}{green}
\defineTColor{Tyellow}{yellow}
\defineTColor{Tred}{red} 
\defineTColor{Tblue}{blue} 
\defineTColor{Tmagenta}{magenta}
\defineTColor{Torange}{orange}
\defineTColor{Twhite}{white}
\defineTColor{Tgray}{lgrey}
\defineTColor{Tbgreen}{brightgreen}

\newcommand{\sn}{\mathop{\delta}\limits^{\doublewedge}}
\newcommand{\sfar}{\mathop{\not{\delta}}\limits_{\mbox{\tiny$\doublevee$}}}
\newcommand{\CL}{\mbox{CL}}
\newcommand{\cl}{\mbox{cl}}
\newcommand{\Int}{\mbox{int}}
\begin{document}

\title[Hit and Miss Hypertopologies]{Strongly Hit and Far Miss  Hypertopology \&\\ Hit and Strongly Far Miss  Hypertopology}

\author[J.F. Peters]{J.F. Peters$^{\alpha}$}
\email{James.Peters3@umanitoba.ca, cguadagni@unisa.it}
\address{\llap{$^{\alpha}$\,}Computational Intelligence Laboratory,
University of Manitoba, WPG, MB, R3T 5V6, Canada and
Department of Mathematics, Faculty of Arts and Sciences, Ad\i yaman University, 02040 Ad\i yaman, Turkey}
\author[C. Guadagni]{C. Guadagni$^{\beta}$}
\address{\llap{$^{\beta}$\,}Computational Intelligence Laboratory,
University of Manitoba, WPG, MB, R3T 5V6, Canada and
Department of Mathematics, University of Salerno, via Giovanni Paolo II 132, 84084 Fisciano, Salerno , Italy}
\thanks{The research has been supported by the Natural Sciences \&
Engineering Research Council of Canada (NSERC) discovery grant 185986.}


\date{}

\dedicatory{Dedicated to the Memory of Som Naimpally}

\begin{abstract}
This article introduces the {\it strongly hit and far-miss $\tau^\doublewedge_\mathscr{B} $ as well as hit and strongly far miss $\tau_{\doublevee, \mathscr{B}}$ hypertopologies on $\CL(X)$ associated with} ${\mathscr{B}}$, a nonempty family of subsets on the topological space $X$. They result from the strong farness and strong nearness proximities. The main results in this paper stem from the Hausdorffness of $(\CL(X), \tau_{\doublevee, \mathscr{B}})$ and $(\mbox{RCL}(X), \tau^\doublewedge_\mathscr{B}  ) $, where $\mbox{RCL}(X)$ is the space of regular closed subsets of $X$. To obtain the results, special local families are introduced.
\end{abstract}


\maketitle

\section{Introduction}

This article introduces some new hypertopologies on the space $\CL(X)$ of non-empty closed subsets and on $\mbox{RCL}(X)$, the space of regular closed subsets of a topological space $X$.  These new hypertopologies result from the introduction of strong farness~\cite{PetersGuadagni2015stronglyFar} and strong nearness~\cite{PetersGuadagni2015stronglyNear}.  Such hypertopologies are located in the class of hit and miss ones. Significant examples of such topologies are Hausdorff, Fell and Attouch-Wets hypertopologies.  Interest in this topic spans many years (see, {\em e.g.}, ~\cite{Beer1993,Beer1993hit,DiConcilio2013action,DiConcilio2000SetOpen,DiConcilio2000PartialMaps,DiConcilio1989,DiMaio2008hypertop,DiMaio1995hypertop,DiMaio1992hypertop,Lucchetti1994,Lucchetti1995,Som2006hypertopology}) with a number of possible applications.  

The \emph{strongly near} proximity~\cite{PetersGuadagni2015stronglyNear} and \emph{strongly far} proximity~\cite{PetersGuadagni2015stronglyFar} and~\cite{Peters2015visibility}, provide a foundation for the hypertopologies introduced, here. Strong nearness plays a role in the new hit sets, while strong farness leads to some new miss sets.  
%We consider such topologies associated with local families of topological spaces.

\section{Preliminaries}
In this work we focus our attention on two new kinds of hit and far-miss topologies. On the one hand, we use the concept of \textit{strong farness} \cite{PetersGuadagni2015stronglyFar} and on the other hypertopology is based on \emph{strongly near proximity}~\cite{PetersGuadagni2015stronglyNear}.

Strong proximities are associated with Lodato proximity and the Efremovi\v c property. 
Recall how a \textit{Lodato proximity} is defined~\cite{Lodato1962,Lodato1964,Lodato1966} (see, also, \cite{Naimpally2009,Naimpally1970,Guadagni2015}).

\begin{definition} 
Let $X$ be a nonempty set. A \textit{Lodato proximity $\delta$} is a relation on $\mathscr{P}(X)$, which satisfies the following properties for all subsets $A, B, C $ of $X$:
\begin{itemize}
\item[P0)] $A\ \delta\ B \Rightarrow B\ \delta\ A$
\item[P1)] $A\ \delta\ B \Rightarrow A \neq \emptyset $ and $B \neq \emptyset $
\item[P2)] $A \cap B \neq \emptyset \Rightarrow  A\ \delta\ B$
\item[P3)] $A\ \delta\ (B \cup C) \Leftrightarrow A\ \delta\ B $ or $A\ \delta\ C$
\item[P4)] $A\ \delta\ B$ and $\{b\}\ \delta\ C$ for each $b \in B \ \Rightarrow A\ \delta\ C$
\end{itemize}
Further $\delta$ is \textit{separated }, if 
\begin{itemize}
\item[P5)] $\{x\}\ \delta\ \{y\} \Rightarrow x = y$.
\end{itemize}
\end{definition}

\noindent $A\ \delta\ B$ reads "$A$ is near to $B$" and $A \not \delta B$ reads "$A$ is far from $B$".
A \emph{basic proximity} is one that satisfies the \v{C}ech axioms $P0)-P3)$~\cite[\S 2.5, p. 439]{Cech1966}.
\textit{Lodato proximity} or \textit{LO-proximity} is one of the simplest proximities. We can associate a topology with the space $(X, \delta)$ by considering as closed sets those sets that coincide with their own closure where, for a subset $A$, we have
\[
\mbox{cl} A = \{ x \in X: x\ \delta\ A\}.
\]
This is possible because of the correspondence of Lodato axioms with the well-known Kuratowski closure axioms. 

By considering the gap between two sets in a metric space ( $d(A,B) = \inf \{d(a,b): a \in A, b \in B\}$ or $\infty$ if $A$ or $B$ is empty ), Efremovi\v c introduced a stronger proximity called \textit{Efremovi\v c proximity} or \textit{EF-proximity}~\cite{Efremovich1951,Efremovich1952}.  

\begin{definition}
An \emph{EF-proximity} is a relation on $\mathscr{P}(X)$ which satisfies $P0)$ through $P3)$ and in addition 
\[A \not\delta B \Rightarrow \exists E \subset X \hbox{ such that } A \not\delta E \hbox{ and } X\setminus E \not\delta B \hbox{ EF-property.}\]
\end{definition}

A topological space has a compatible EF-proximity if and only if it is a Tychonoff space.

\setlength{\intextsep}{0pt}
\begin{wrapfigure}[10]{R}{0.45\textwidth}
%\fbox{
\begin{minipage}{4.5 cm}
\begin{center}
\begin{pspicture}
 %[showgrid=true]
 (0.0,3.5)(2.5,4.0)
%\psset{linewidth=0.5pt,linecolor=blue}
\psframe[linecolor=black](-0.8,0.5)(4.5,4.0)
\pscircle[linestyle=dotted,dotsep=0.05,linecolor=black,linewidth=0.05,style=Torange](0.18,1.55){0.88}
\pscircle[linestyle=dotted,dotsep=0.05,linecolor=black,linewidth=0.05,style=Tyellow](0.98,1.35){0.78}
\pscircle[linecolor=black,linestyle=solid,linewidth=0.05,style=Tgray](3.38,1.85){1.00}
\pscircle[linestyle=dotted,dotsep=0.05,linecolor=black,linewidth=0.05,style=Tyellow](2.38,2.55){0.70}\pscircle[linestyle=solid,dotsep=0.05,linecolor=black,linewidth=0.05,style=Torange](3.52,3.05){0.50}
\rput(-0.5,3.8){\footnotesize  $\boldsymbol{X}$}
\rput(1.40,1.35){\footnotesize $\boldsymbol{\Int A}$}
\rput(0.08,2.15){\footnotesize $\boldsymbol{\Int B}$}
\rput(3.68,1.65){\footnotesize $\boldsymbol{E}$}
\rput(3.42,3.05){\footnotesize $\boldsymbol{D}$}
\rput(2.08,2.55){\footnotesize  $\boldsymbol{C}$}
\rput(0.28,0.15){\qquad\qquad\qquad\qquad\footnotesize 
                 $\boldsymbol{\mbox{Fig.} 2.1.\  A\ \sn\ B,\  E\ \sn\ (C \cup D)}$}\label{fig:stronglyNear}
 \end{pspicture}
%\caption{\footnotesize $\boldsymbol{A\sn B, C\sn E}$}
%\label{fig:stronglyNear}
\end{center}
\end{minipage}
%}
\end{wrapfigure}
\setlength{\intextsep}{2pt}

Any proximity $\delta$ on $X$ induces a binary relation over the powerset exp $X,$ usually denoted as $ \  \ll_\delta $ and  named    the  {\it   natural strong inclusion associated with } $\delta,$ by declaring that $ A$ is {\it strongly included} in $B, \ A \ll_{\delta} B, \ $ when $A$ is far from the complement of $B,  \ A \not\delta X \setminus B .$

By strong inclusion the \textit{Efremivi\v c property} for $ \delta $ can be written also as a betweenness property   \

 \centerline {  (EF) \  \   \  \  If $A \ll_{\delta} B,$  then there exists some $C$ such that $A \ll_{\delta} \  C \ll_{\delta} \ B$.} \  \ \  

We say that $A$ and $B$ are \textbf{\emph{$\delta-$strongly far}}~\cite{PetersGuadagni2015stronglyNear}, where $\delta$ is a Lodato proximity, and we write $\mathop{\not{\delta}}\limits_{\mbox{\tiny$\doublevee$}}$ if and only if $A \not\delta B$ and there exists a subset $C$ of $X$ such that $A \not\delta X \setminus C$ and $C \not\delta B$, that is the Efremovi\v c property holds on $A$ and $B$.

Instead, the concept of \textbf{\emph{strongly near}} proximity arises from the need to introduce a relation yields information about the interiors of pairs of subsets that at least have non-empty intersection.
We say that the relation $\sn$ on $\mathscr{P}(X)$ is an \emph{strongly near proximity}, provided it satisfies the following axioms.  Let $A, B, C \subset X$ and $x \in X$.
\begin{itemize}
\item[N0)] $\emptyset \not\sn A, \forall A \subset X $, and \ $X \sn A, \forall A \subset X$
\item[N1)] $A \sn B \Leftrightarrow B \sn A$
\item[N2)] $A \sn B \Rightarrow A \cap B \neq \emptyset$
\item[N3)] If $\Int(B)$ and $\Int(C)$ are not equal to the empty set, $A\ \sn\ B$ or $A\ \sn\ C \ \Rightarrow \ A \sn (B \cup C)$
\item[N4)] $\mbox{int}A \cap \mbox{int} B \neq \emptyset \Rightarrow A\ \sn\ B$
\end{itemize}

\begin{example} {\bf Intersecting Interiors}.\\
In Fig. 2.1, $A \sn B$ (axiom N4) and $E\ \sn\ (C \cup D)$  (axiom N3).
\qquad \textcolor{blue}{$\blacksquare$}
\end{example}

When we write $A \sn B$, we read $A$ is \emph{strongly near} to $B$~\cite{Peters2015visibility,Peters2015stronglyNear} (see, also,~\cite{PetersGuadagni2015stronglyNear}).
For each \emph{almost proximity} we assume the following relations:
\begin{itemize}
\item[N5)] $x \in \Int (A) \Rightarrow \{x\} \sn A$
\item[N6)] $\{x\} \sn \{y\} \Leftrightarrow x=y$  \qquad \textcolor{blue}{$\blacksquare$}
\end{itemize}
%\[
%(N5) \quad \ \ x \in \Int (A) \Rightarrow x \sn A; \ \ 
%(N6)\quad \{x\} \sn \{y\} \Leftrightarrow x=y  \qquad \textcolor{blue}. 
%\] 
If we take the almost proximity related to non-empty intersection of interiors, we have that $A \sn B \Leftrightarrow \Int A \cap \Int B \neq \emptyset$ provided $A$ and $B$ are not singletons; if $A = \{x\}$, then $x \in \Int(B)$, and if $B$ too is a singleton, then $x=y$. If $A \subset X$ is an open set, then each $x\in A$ is strongly near $A$.

\section{New hypertopologies}
Let $\CL(X)$ be  the hyperspace of all non-empty closed subsets of  a space  $X.$
 {\it Hit and miss} and {\it hit and far-miss}  topologies on $\CL(X)$  are obtained by the join of two halves. Well-known examples are Vietoris topology~\cite{Vietoris1921,Vietoris1922,Vietoris1923,Vietoris1927} (see, also,~\cite{Attouch1991,Beer1993,Beer1993hit,DiConcilio2013action,DiConcilio2000SetOpen,DiConcilio2000PartialMaps,DiConcilio1989,DiMaio2008hypertop,DiMaio1995hypertop,DiMaio1992hypertop,Som2006hypertopology,Guadagni2015}) and Fell topology~\cite{Fell1962HausdorfTop}.
In \cite{PetersGuadagni2015stronglyFar} and \cite{PetersGuadagni2015stronglyNear}, the following new hypertopologies are introduced.\\
\vspace{2mm}

\noindent $\clubsuit$ \ \ $\tau_\doublevee$ is the \emph{hit and strongly far-miss} topology having as subbase the sets of the form:
\begin{itemize}
\item $V^- = \{E \in {\rm \mbox{CL}}(X): E \cap V \neq \emptyset\}$, where $V$ is an open subset of $X$,
\item $A_\doublevee =  \{ \ E \in \mbox{CL}(X) : E \stackrel{\not{\text{\normalsize$\delta$}}}{\text{\tiny$\doublevee$}} X\setminus  A  \ \}$, where $A$ is an open subset of $X$
\end{itemize}
\medskip
\noindent $\clubsuit$ \ \ $\tau^\doublewedge$ is the \emph{strongly hit and far-miss} topology having as subbase the sets of the form:
\begin{itemize}
\item $V^{\doublewedge} = \{E \in \CL(X): E \sn V \}$, where $V$ is an open subset of $X$,
\item $A^{++} =  \{ \ E \in \CL(X) : E \not\delta X\setminus  A  \ \}$, where $A$ is an open subset of $X$,
\end{itemize}

\noindent where in both cases $\delta$ is a Lodato proximity compatible with the topology on $X$.

It is possible to consider several generalizations. For example, for the miss part we can look at subsets running in a family of closed sets $\mathscr{B}$. So we define the {\it strongly hit and far-miss topology on $\CL(X)$ associated with} ${\mathscr{B}}$ as the topology generated by the join of the hit sets $ A^{\doublewedge},$  where  $A$ runs over all  open subsets of $X$, with the miss sets $A^{++}$, where   $A$ is once again an open subset of $X,$ but more,  whose  complement runs in   $\mathscr{B}$. We can do the same with the \textit{hit and strongly far-miss} topology.

\section{Main Results}
Next, consider \emph{ strongly hit and far miss topologies } and \emph{hit and strongly far miss topologies} associated with families of subsets. We look at conditions that make these topologies $T_2$ topologies. 

\subsection{Hit and strongly far miss topologies} $\mbox{}$\\
Here, we consider results contained in \cite{Beer1993}, {\em e.g.}, $\tau_{\doublevee, \mathscr{B}}$ is the topology having as subbase the sets of the form:

\begin{itemize}
\item $V^- = \{E \in {\rm \mbox{CL}}(X): E \cap V \neq \emptyset\}$, where $V$ is an open subset of $X$,
\item $A_\doublevee =  \{ \ E \in \mbox{CL}(X) : E \stackrel{\not{\text{\normalsize$\delta$}}}{\text{\tiny$\doublevee$}} X\setminus  A  \ \}$, where $A$ is an open subset of $X$ and $X \setminus A \in \mathscr{B}$
\end{itemize}

\begin{definition}
Let $X$ be a topological space and $\mathscr{B}$ a non-empty family of subsets of $X$. We call $\mathscr{B}$ a \emph{strongly local family} if and only if 
\[\forall x \in X \hbox{ and } \forall U \hbox{ nbhd of }x, \ \exists S \in \mathscr{B} : x \in \Int(S) \subseteq S \ll_\doublevee U,\]
where $ S \ll_\doublevee U$ means $S \sfar X \setminus U$.  \qquad \textcolor{blue}{$\blacksquare$}
\end{definition}

We write $\Sigma(\mathscr{B})$ for the collection of all finite unions of elements of $\mathscr{B}$.

%\setlength{\intextsep}{0pt}
%\begin{wrapfigure}[9 ]{R}{0.45\textwidth}
%%\fbox{
%\begin{minipage}{4.5 cm}
\begin{figure}[!ht]
\begin{center}
\begin{pspicture}
 %[showgrid=true]
 (1.5,0.0)(2.5,4.5)
%\psset{linewidth=0.5pt,linecolor=blue}
\psframe[linecolor=black](-0.8,0.5)(3.5,4.0)
%\pscircle[linestyle=dotted,dotsep=0.05,linecolor=black,linewidth=0.05,style=Torange](0.18,1.55){0.88}
%\pscircle[linestyle=dotted,dotsep=0.05,linecolor=black,linewidth=0.05,style=Tyellow](0.98,1.35){0.78}
\pscircle[linecolor=black,linestyle=dotted,linewidth=0.05,style=Tgray](1.38,2.20){1.65}
\pscircle[linestyle=solid,dotsep=0.05,linecolor=black,linewidth=0.05,style=Tyellow](1.38,2.20){0.82}
\rput(-0.5,3.8){\footnotesize  $\boldsymbol{X}$}
\rput(1.45,2.20){\footnotesize $\boldsymbol{\bullet x}$}
%\rput(1.40,1.35){\footnotesize $\boldsymbol{\Int A}$}
%\rput(0.08,2.15){\footnotesize $\boldsymbol{\Int B}$}
\rput(1.38,3.55){\footnotesize $\boldsymbol{U_x}$}
\rput(1.38,2.85){\footnotesize  $\boldsymbol{A}$}
\rput(0.28,0.15){\qquad\qquad\qquad\qquad\footnotesize 
                 $\boldsymbol{\mbox{Fig.}\ 3.1\  A \ll_{\delta} U_x}$}
								%\label{fig:stronglyNear2}
 \end{pspicture}
%\capton{\footnotesize $\boldsymbol{A\sn B, C\sn E}$}
%\label{fig:stronglyNear}
%\caption{Something}
%\label{fig:nested2}
\end{center}
\end{figure}
%\end{center}
%\end{minipage}
%%}
%\end{wrapfigure}
%\setlength{\intextsep}{2pt}

\begin{example} {\bf Strongly Local Family 1}.\\
A first simple example of strongly local family is the following. Take $X$ as a locally compact topological space, $\mathscr{B}$ the family of compact subsets and $\delta$ as the \emph{Alexandroff proximity} defined by \ $A\ \delta_A\ B \Leftrightarrow \mbox{cl}A \cap \mbox{cl}B \neq \emptyset$ or both $\mbox{cl}A$ and $\mbox{cl}B$ are non-compact. In fact, in this case, for each $x \in X$ and each nbhd $U$ of $x$, it is always possible to find a compact subset containing $x$ and strongly contained in $U$.  See, {\em e.g.}, Fig. 3.1.
\qquad \textcolor{blue}{$\blacksquare$}
\end{example}

%\setlength{\intextsep}{0pt}
%\begin{wrapfigure}[9 ]{R}{0.48\textwidth}
%%\fbox{
%\begin{minipage}{4.9 cm}
\begin{figure}[!ht]
\begin{center}
\begin{pspicture}
 %[showgrid=true]
 (0.0,0.0)(2.5,4.5)
%\psset{linewidth=0.5pt,linecolor=blue}
\psframe[linecolor=black](-1.2,0.3)(4.8,4.55)
\pscircle[linecolor=black,linestyle=solid,dotsep=0.05,linewidth=0.05](1.52,2.55){1.91}
%\pscircle[linestyle=solid,dotsep=0.05,linecolor=black,linewidth=0.05,style=Tgray](1.52,2.05){1.08}
\pscircle[linestyle=solid,dotsep=0.05,linecolor=black,linewidth=0.05,style=Tgreen](1.52,2.55){0.38}
\pswedge[linestyle=dotted,dotsep=0.05,style=Tgray](0.52,1.85){1.8}{10}{60}
\psellipse[linecolor=black,linestyle=dotted,dotsep=0.05,linewidth=0.05,style=Tgray](1.52,2.55)(0.45,1.80)
\psellipse[linecolor=black,linestyle=dotted,dotsep=0.05,linewidth=0.05,style=Tgray](1.52,2.55)(1.55,1.25)
%\psline[linestyle=solid,linecolor=black,linewidth=0.05]{-}(0.52,1.95)(0.52,3.50)
\rput(-0.8,4.3){\footnotesize  $\boldsymbol{X}$}
\rput(1.52,2.55){\tiny $\boldsymbol{A}$}
\rput(0.52,3.85){\footnotesize $\boldsymbol{C}$}
\rput(0.32,2.85){\footnotesize $\boldsymbol{B_1}$}
\rput(1.52,1.05){\tiny $\boldsymbol{B_2}$}
\rput(0.82,2.05){\tiny $\boldsymbol{B_3}$}
\rput(0.28,0.00){\qquad\qquad\qquad\qquad\footnotesize 
                 $\boldsymbol{\mbox{Fig.}\ 3.2.\  A \ll_{\delta} B_i \ll_{\delta}  C}$}
 \end{pspicture}
%\caption{\footnotesize $\boldsymbol{A\sn B, C\sn E}$}
%\label{fig:stronglyNear}
\end{center}
\end{figure}
%\end{minipage}
%%}
%\end{wrapfigure}
%\setlength{\intextsep}{2pt}

\begin{example} {\bf Strongly Local Family 2}.\\
A more general example is given by local proximity spaces $(X, \delta, \mathscr{B})$ \cite{Leader1967}, where $\delta$ is a proximity generally non-Efremovi\v c and $\mathscr{B}$ is a boundedness. In particular this family satisfies the following property:  if $A \in {\mathscr{B}}, \  C \subset X$ and $A \ll_{\delta} C$  then there exists some $B \in {\mathscr{B}}$ such that $A \ll_{\delta} B \ll_{\delta}  C$, where $\ll_{\delta}$ is the  natural strong inclusion  associated with $\delta.$ By this property follows that $\mathscr{B}$ is a strongly local family.  See, {\em e.g.}, Fig. 3.2.
\qquad \textcolor{blue}{$\blacksquare$}
\end{example}

\begin{theorem}\label{T2far}
Let $(X, \tau)$ be a topological space, $\delta$ a compatible Lodato proximity on $X$ and $\mathscr{B}$ a non-empty subfamily of $\CL(X)$ . Suppose that any point that does not belong to a closed set is strongly far from the closed set. $(\CL(X), \tau_{\doublevee, \mathscr{B}}  ) $ is $T_2$ if and only if $\Sigma(\mathscr{B})$ is a strongly local family.
\end{theorem}
\begin{proof}
$"\Leftarrow"$.We want to prove that $(\CL(X), \tau_{\doublevee, \mathscr{B}}  ) $ is $T_2$. Suppose that $A$ and $B$ are distinct closed subsets of $X$. So there exists a point $b \in B \cap (X \setminus A)$ and by the assumption there exists a subset $S \in \mathscr{B}$ such that $x \in \Int(S) \subset S \ll_{\doublevee} X \setminus A$. Then $B \in (\Int S)^-$ and $A \in (X\setminus S)_\doublevee$ with $(\Int S)^- \cap (X\setminus S)_\doublevee = \emptyset$.\\
$"\Rightarrow"$. Consider now $x \in X$ and suppose $X \setminus A$ to be a nbhd of $x$. We have to prove that the special betweenness property holds. Take $A$ and $A \cup \{x\}$. By the hypothesis we know that there exist two open sets  $\mathscr{A}_1, \mathscr{A}_2 \in \tau_{\doublevee, \mathscr{B}}$ such that $A \in \mathscr{A}_1, \ A\cup \{x\} \in \mathscr{A}_2$ and $\mathscr{A}_1 \cap \mathscr{A}_2= \emptyset$. Suppose
\begin{center}
$ \mathscr{A}_1= (X \setminus S)_\doublevee \cap (\bigcap_{i=1}^n V_i^-)$\\
 $\mathscr{A}_2= (X \setminus T)_\doublevee \cap (\bigcap_{j=1}^m W_j^-)$\\
\end{center}
 with $S , \ T \in \mathscr{B}$ and $V_1,...,V_n$, $W_1,..., W_m$ open subsets of $X$. Further we may assume that $S \cap V_i= \emptyset,\ \forall i=1,...,n$ and $T \cap W_j = \emptyset, \ \forall j=1,...,m.$

 Consider now $\{l_1,..., l_k\}$ indices for which $x \in W_l$. We want to show that for at least one of these indices $W_l \subseteq S, \ l \in \{l_1,..., l_k \} $. In fact if, by contradiction, $W_l \not\subseteq S \ \forall l \in \{l_1,..., l_k\}$, then there would exist $x_l \in W_l \cap (X \setminus S) \ \forall l \in \{l_1,..., l_k\}.$ But now $A \cup \{x_{l_1},..., x_{l_k}\} \in \mathscr{A}_1 \cap \mathscr{A}_2$, and this is absurd. Hence there exists $W_{l_{\overline{j}}} \subseteq S$ and we have $x \in W_{l_{\overline{j}}} \subseteq \Int(S) \subset S \ll_\doublevee X \setminus A$.
\end{proof}

\setlength{\intextsep}{0pt}
\begin{wrapfigure}[14]{R}{0.45\textwidth}
\begin{minipage}{4.5 cm}
\begin{center}
\psset{viewpoint=20 10 40 rtp2xyz,lightsrc=viewpoint,Decran=28,ngrid=36 36,grid=false}
\begin{pspicture}(-2,2)(2,2)
\psSolid[object=sphere,r=1.45,lightintensity=1.25]
\psSolid[object=sphere,r=0.6,lightintensity=1.0]
\rput(-0.2,0){\tiny $\boldsymbol{p}$}
\rput(0.2,0.5){\footnotesize $\boldsymbol{C}$}
\rput(0.2,1.8){\footnotesize $\boldsymbol{U}$}
\axesIIID[showOrigin=true](2,2,2)(2.1,2.1,2.1)
\end{pspicture}
\rput(-3.5,-5){\qquad\qquad\qquad\qquad\footnotesize 
                 $\boldsymbol{\mbox{Fig.}\ 3.3.\ p\ \in\ \Int(C) \subseteq C \subset U}$}\label{fig:convex}
\end{center}
\end{minipage}
\end{wrapfigure}
\setlength{\intextsep}{2pt}

\subsection{Strongly hit and far miss topologies} $\mbox{}$\\
%Consider now the situation with \emph{ strongly hit and far miss topologies }. 
In this case, we want consider the family of \emph{regular closed subsets of $X$}, $\mbox{RCL}(X)$. Recall that a set $F$ is \emph{regular closed} if $F = \cl{(\Int F)}$, that is $F$ coincides with the closure of its interior. Considering the nature of almost proximities, this family seem to be the most suitable to which refer. A well-known fact is that regular closed sets form a complete Boolean lattice \cite{Ronse1990}. Moreover there is a one-to-one correspondence between regular open ($\mbox{RO}(X)$) and regular closed sets. We have a \emph{regular open} set $A$ when $A = \Int{(\cl A)}$, that is $A$ is the interior of its closure. The correspondence between the two mentioned classes is given by $c: \mbox{RO}(X) \rightarrow \mbox{RCL}(X)$, where $c(A)= \cl(A)$, and $o: \mbox{RCL}(X) \rightarrow \mbox{RO}(X)$, where $o(F) = \Int (F)$. By this correspondence it is possible to prove that also the family of regular open sets is a complete Boolean lattice. Furthermore it is shown that every complete Boolean lattice is isomorphic to the complete lattice of regular open sets in a suitable topology.\\
The importance of these families is also due to the possibility of using them for digital images processing, because they allow to satisfy certain common-sense physical requirements.\\

Consider now $\tau^\doublewedge_\mathscr{B}$, the \emph{strongly hit and far-miss} topology associated to $\mathscr{B}$:

\begin{itemize}
\item $V^{\doublewedge} = \{E \in \mbox{RCL}(X): E \sn V \}$, where $V$ is a regular open subset of $X$,
\item $A^{++} =  \{ \ E \in \mbox{RCL}(X) : E \not\delta X\setminus  A  \ \}$, where $A$ is a regular open subset of $X$ and $X \setminus A \in \mathscr{B}$.
\end{itemize}

It is easy to prove that this family generates a topology on $\mbox{RCL}(X)$.

\begin{lemma}\label{hitregular}
Let $(X, \tau)$ be a topological space and $\delta^\doublewedge$ strongly near proximity on $X$. For each regular open set $V$on $\mbox{RCL}(X)$, we have that $V^{\doublewedge} = \{E \in \mbox{RCL}(X): E \cap V \neq \emptyset \}$.
\end{lemma}
\begin{proof}
For each regular open set $V$ in $X$, it is straightforward that $V^\doublewedge \subseteq \{E \in \mbox{RCL}(X): E \cap V \neq \emptyset \}$. Look now at the other inclusion and suppose that $A \not\sn V$.  Hence, by property $(N4)$, we have that $\Int A \cap V = \emptyset$. But for $\mbox{RCL}(X)$, it follows that $A \cap V = \emptyset$. In fact, if this intersection is not empty, we can find an element $a \in A = \cl(\Int A)$ such that it is approximated by a net of elements in $\Int A$, $\{a_\lambda\}_{\lambda \in \Lambda}$. Now, since $V$ is a nhbd of $a$, $V$ must contain $\{a_\lambda\}_{\lambda \in \Lambda}$ residually. But this is absurd. 
\end{proof}

\begin{definition}
Let $X$ be a topological space and $\mathscr{B}$ a non-empty family of subsets of $X$. We call $\mathscr{B}$ a \emph{regular local family } if and only if 
\[\forall x \in X \hbox{ and } \forall U \hbox{regular open set containing }x, \ \exists S \in \mathscr{B} : x \in \Int(S) \subseteq S \ll_\delta U,\]
where $ S \ll_\delta U$ means $S \not\delta X \setminus U$.
\end{definition}

\begin{theorem}
Let $(X, \tau)$ be a topological space, $\sn$ an almost proximity on $X$, $\delta$ a compatible Lodato proximity on $X$ and $\mathscr{B}$ a non-empty subfamily of $\mbox{RCL}(X)$ . If $\Sigma(\mathscr{B})$ is a regular local family, then $(\mbox{RCL}(X), \tau^\doublewedge_\mathscr{B}  ) $ is $T_2$.
\end{theorem}
\begin{proof}
The result simply follows by applying mostly the same procedure of thm. \ref{T2far} and by lemma \ref{hitregular}.
\end{proof}

\begin{remark}
Observe that in particular we could refer these results to a generalization of the Fell topology if we take as family $\mathscr{B}$ that one of compact subsets of a topological space.
\end{remark}

\begin{example} {\bf Regular local family}.\\
An easy example of regular local family contained in $\mbox{RCL}(X)$ is obtained by considering $X = \mathbb{R}^N$ in which the family of closed convex subsets is a local family, {\em i.e.},
\[
\forall x \in X \hbox{ and } \forall U \hbox{nbhd of }x,  \exists C \hbox{ convex }: x \in \Int(C) \subseteq C \subset U. 
\]
See, {\em e.g.}, Fig. 3.3 in the 3D spherical nbhd $C$ of $p$ is entirely in the interior of the nbhd $U$ of $p$.  An important thing to notice is that in this case each open convex set is also regular open.
\qquad \textcolor{blue}{$\blacksquare$}
\end{example}

\end{document}